\documentclass{article}[11]

\usepackage{fullpage, amsmath,amssymb, amsthm, mathrsfs,verbatim}\usepackage{xspace}
\usepackage{xcolor}
\usepackage{authblk}

\usepackage{geometry}
\usepackage{afterpage}
\usepackage{lipsum}

\definecolor{verde}{rgb}{0.,0.6,0.2}
\definecolor{bianco}{rgb}{1.,1.,1.}
\definecolor{marrone}{rgb}{0.7,0.2,0.1}
\definecolor{rosso}{rgb}{1,0,0}
\definecolor{giallo}{rgb}{1.0, 0.87, 0.0}
\definecolor{blu}{rgb}{0.03, 0.27, 0.49}
\definecolor{daffodil}{rgb}{0.03, 0.27, 0.49}
\definecolor{darkcerulean}{rgb}{1.0, 0.87, 0.0}

\usepackage{tikz}

\usetikzlibrary{matrix}
\usetikzlibrary{positioning}
\usetikzlibrary{shapes.geometric}
\usetikzlibrary{shapes,decorations,shadows}  
\usetikzlibrary{decorations.pathmorphing}  
\usetikzlibrary{decorations.shapes}  
\usetikzlibrary{fadings}  
\usetikzlibrary{patterns}  
\usetikzlibrary{calc}  
\usetikzlibrary{automata}
\usetikzlibrary{chains,fit,shapes}
\usetikzlibrary{backgrounds}

\newtheorem{theorem}{Theorem}[section]
\newtheorem{definition}[theorem]{Definition}

\newtheorem{lemma}[theorem]{Lemma}

\newtheorem{proposition}[theorem]{Proposition}
\newtheorem{remark}[theorem]{Remark}
\newtheorem{example}[theorem]{Example}

\usepackage{hyperref}
\usepackage{eufrak}

\renewenvironment{abstract}{%
  \noindent\bfseries\abstractname.\normalfont}{}

\begin{document}

\newgeometry{a4paper,
 total={170mm,257mm},
  left=30mm,
 right=30mm,
 top=55mm,
 bottom=35mm
 }

\title{{\textbf{ Kähler Immersions of ALE Kähler Manifolds into
Complex Space Forms}}}

\author[1]{Farnaz Ghanbari} 
\author[2] {Abbas Heydari\footnote{The corresponding author, aheydari@modares.ac.ir}}
\affil[1]{Tarbiat Modares University, Tehran, Iran and ICTP, Trieste, Italy}
\affil[2]{Tarbiat Modares University, Tehran, Iran}

\date{}
\maketitle 

\hspace{1.4cm}\begin{minipage}[b]{0.77\linewidth}

\begin{abstract}
In this article, we investigate the Kähler immersions of special Asymptotically Locally Euclidean (ALE) Kähler metrics into complex space forms. We provide a relation between Kähler immersions problem of these metrics and the sign of their mass. Our result shows that these metrics with positive mass do not admit a Kähler immersion into complex Euclidean space. Moreover, we also obtain that these metrics with any sign of mass cannot be Kähler immersed into complex hyperbolic space.
\end{abstract}
\end{minipage}

\section{Introduction}
\label{intro}

\hspace{10pt} There exists a well-known theorem in Riemannian geometry by John Nash, which states that any Riemannian manifold can be isometrically immersed into the real Euclidean space $\mathbb{R}^{N}$ for sufficiently large $N$. In contrast to the Riemannian case, a Kähler manifold does not always admit a Kähler immersion into the complex Euclidean space $\mathbb{C}^{N}$ (not even if $N$ is infinite). The study of Kähler immersions between Kähler manifolds was started by Eugenio Calabi. Calabi established an algebraic criterion to determine when a Kähler manifold $(M, g)$ admits a Kähler (i.e., holomorphic and isometric) immersion into a finite or infinite-dimensional complex space form \cite{Calabi}. Calabi’s initial observation was that if a Kähler immersion of $(M, g)$ into a complex space form exists, then the metric $g$ is forced to be real analytic being the pull-back via a holomorphic map of the real analytic metric of a complex space form \cite{Calabi}. He introduced the diastasis function (special unique Kähler potential) which plays a key role in studying Kähler immersions of a Kähler manifold into another Kähler manifold. In practice, the diastasis function is not always explicitly provided, and while Calabi’s criterion is theoretically impeccable, it often poses challenges during application. Nevertheless, over the past six decades, numerous mathematicians have worked on the subject and many interesting results have been obtained (\cite{Loi1, Loi2, Loi4, Loi3, Filippo, Scala, Zedda}).

This work contributes to the research line of Kähler immersions problem, to be specific Kähler immersions of special ALE Kähler metrics into complex space forms. We need to compute the diastasis associated to these metrics and investigate the Kähler immersions problem into complex space forms such as complex Euclidean space, complex projective space and complex hyperbolic space. Therefore, it is required to choose a point ($p$), write a diastasis function centered at $p$, consider the power expansion of the diastasis function around the origin and compute all the coefficients in its power expansion, then we need to use Calabi's criterion to discuss about this pro-
\restoregeometry
\hspace{-15pt}blem. The main reason to explore and study Kähler immersions problem for these spaces is because they are interesting in General Relativity. One of the most interesting problems in General Relativity is related to the mass which is important for physicists. This concept originated in General Relativity, where an asymptotically flat $3$-manifold could be interpreted as representing a time-symmetric slice of some $4$-dimensional space-time, in which case this invariant becomes the so-called ADM mass \cite{Mass}. In \cite{Hein} the authors established a straightforward and explicit formula for the mass of any ALE Kähler manifold. This result assumes only the type of weak fall-off conditions necessary for the mass to be well-defined. We establish a connection between Kähler immersions problem of these metrics and the sign of their mass. Our main result demonstrates that these metrics with positive mass do not admit a Kähler immersion into complex Euclidean space. Furthermore, we  prove that these metrics, regardless of their mass sign, cannot be Kähler immersed into complex hyperbolic space.
However, in the specific case of complex projective space, no conclusive results have been obtained.

The rest of the paper is organized as follows. Section \ref{sec:prel} provides essential definitions, some well-known theorems, and key results that will be used throughout the paper. In Section \ref{sec:sec1}, we state and prove the main theorems.

\section{Preliminaries}
\label{sec:prel}
\hspace{10pt} We devote this section to recalling the main definitions and results on Kähler immersions and mass in Kähler geometry, which will be essential for our subsequent discussions. We recommend the references (\cite{Calabi, Hein, Joyce, Loi})  for a comprehensive overview of the topics mentioned.

A Kähler immersion $f: (M,g) \longmapsto (S,gs)$ from a Kähler manifold $(M,g)$ into a complex space form $(S,gs)$ is a holomorphic map such that $f^{*} gs=g$ (here $g$ and $gs$ denote the Kähler metrics on $M$ and $S$ respectively.)

Recall that a complex space form, which we assume to be complete and simply connected, can be of three types based on the sign of the constant holomorphic sectional curvature: complex Euclidean space, complex projective space, and complex hyperbolic space.

   \begin{itemize}
        \item[1.]The complex Euclidean space $\mathbb{C}^{N}$ of complex dimension $N \leq \infty$, endowed with the flat metric. Here $\mathbb {C}^{\infty}$ denotes the Hilbert space $\ell ^{2}(\mathbb{C})$ consisting of sequences $w_{j}$, $j=1,2,... , w_{j}\in \mathbb{C}$ such that $\Sigma_{j=1}^{+\infty} |w_{j}|^{2}< + \infty$.\,
        \item[2.]The complex projective space $\mathbb{C}P^{N}$ of complex dimension $N \leq  \infty$, with the Fubini-Study metric $g_{FS}$.
           \item[3.]The complex hyperbolic space $\mathbb{C}H^{N}$ of complex dimension $N \leq  \infty$, that is the unit ball $B \subset \mathbb{C}^{N}$given by $B=\Bigl\{(z_{1},...,z_{N})\in \mathbb{C}^{N}, \sum_{j=1}^{N} |z_{j}|^{2} <1 \Bigl\}$ endowed with the hyperbolic metric.
    \end{itemize}

Let $M$ be an $n$-dimensional complex manifold endowed with a real analytic Kähler metric $g$. Recall that the Kähler metric $g$ is real analytic if we fix a local coordinate system $z = (z_1, . . . , z_n)$ on a neighbourhood $U$ of any point $p \in M$; it can be described on $U$ by a real analytic Kähler potential $\Phi : U \longrightarrow \mathbb{R} $. In that case, the potential $\Phi(z)$ can be analytically continued to an open neighbourhood $W \subset U \times U$ of the diagonal. This extension is denoted by $\Phi(z,\Bar{w})$.

The diastasis function associated with the metric $g$ provides the criterion, as stated by Calabi. Below, we present the definition of the diastasis function.

\begin{definition}
    The diastasis function $D(z,w)$ on $W$ is defined by:
    \[
   D(z,w)= \Phi(z,\Bar{z})+  \Phi(w,\Bar{w})- \Phi(z,\Bar{w})- \Phi(w,\Bar{z}). 
    \]
\end{definition}

The diastasis function $D(z,w)$ has some fundamental properties. The following proposition describes these fundamental properties:

\begin{proposition}(\cite{Calabi})
The diastasis function D(z, w) satisfies the following properties:

   \begin{itemize}
        \item[1.]It is uniquely determined by the Kähler metric g and it does not depend on
the choice of the local coordinate system;
\,
        \item[2.] It is real valued in its domain of (real) analyticity;\,
         \item[3.]  It is symmetric in z and w and D(z, z) = 0;\,
           \item[4.] Once fixed one of its two entries, it is a Kähler potential for $g$.
    \end{itemize}

\end{proposition}

Once $p$ is fixed and $z$ is the coordinate around it, $D(p, q) = D_p(z)$ is a Kähler potential in a neighborhood of $p$. The existence of a local Kähler immersion into a complex space form depends only on the derivatives with respect to $z$ and $\Bar{z}$ of $D_{p}(z)$, evaluated at $p$. In particular, if $p$ is the origin of the selected coordinate system, we denote $D_{0}(z)$.

In the following, we employ a multi-index notation: $z^{m_j} = z_1^{m_{j,1}} \ldots z_n^{m_{j,n}}$, where the strings $m_j = (m_{j,1}, \ldots, m_{j,n})$ are ordered such that $|m_j| \leq |m_k|$ for $j < k$.

\begin{theorem}(\cite{Calabi})
    Among all the Kähler potentials, the diastasis $D_{p}(z)$ is characterized by the fact that in every coordinate system $(z_{1}, ..., z_{n})$ centered at $p$, the $\infty \times \infty$ matrix of the coefficients $(a_{jk})$ in its power expansion around the origin 
    \[
    D_{p}(z)= \sum_{j,k=0}^{\infty} a_{jk}z^{m_{j}}\bar{z}^{m_{k}},
    \]
    satisfies $a_{j0}=a_{0j}=0$ for every nonnegative integer $j$.
\end{theorem}
In the following, we introduce the Calabi's criterion for Kähler immersion into the complex flat space $\mathbb{C}^{N}$. 
\begin{definition}
\label{def2:4}
We say that a complex manifold $(M,g)$ admits a local Kähler immersion into $\mathbb{C}^{N}$ if  given any point $p \in M$, there exists a neighbourhood $U$ of $p$ and a map $f: U \rightarrow \mathbb{C}^{N} $ such that:

\begin{itemize}
        \item[1.] f is holomorphic;
        
        \item[2.] f is isometric, i.e., $D_{p}^{M}(z)= \Sigma_{j=1}^{N} |f_{j}(p) - f_{j}(z)|^{2} $; 
        
        \item[3.] There exists $0< R< +\infty$ such that $\Sigma_{j=1}^{N} |f_{j}(z)|^{2} < R$.
    \end{itemize}

\end{definition}

\begin{definition}
    A real analytic Kähler manifold $(M,g)$ is resolvable of rank $N$ at $p \in M$ if $(a_{jk})$ is positive semidefinite of rank $N$.
\end{definition}
In the context of the Kähler immersion problem, several essential theorems exist. We utilize these theorems in the proof of our main results, which are presented in the remainder of this section.

Calabi's criterion for local Kähler immersion into $\mathbb{C}^{N}$ can be stated as follows  \cite[pages 9, 18]{Calabi}.
\begin{theorem}
\label{theo2:6}
    Let $(M,g)$ be a real analytic Kähler manifold. There exists a neighbourhood $U$ of a point $p$ that admits a Kähler immersion into $\mathbb{C}^{N}$  if and only if $(M,g)$ is resolvable of rank at most $N$ at $ p\in M$. Furthermore if the rank is exactly $N$, the immersion is full.
\end{theorem}

The following theorem states that if there exists a local Kähler immersion around a point $p\in M$, then the same holds true for any other point \cite{Calabi}. This conclusion implies that a manifold is resolvable without specifying a particular point.

\begin{theorem}(Global character of resolvability).
\label{theo2:7}
If a real analytic connected Kähler manifold $(M,g)$ is resolvable of rank $N$ at a point $p \in M$, then it also is at any other point.
    
\end{theorem}
Let $F(N, b)$ be an $N$-dimensional complex space form of holomorphic sectional curvature $4b$. we denote by $D_{0}(z)$ the diastasis of $g$ at the origin and consider the power expansion around the origin of the function $(e^{bD_{0}(z)}-1)/b$:
\[
\frac{e^{bD_{0}(z)}-1}{b}= \sum_{j,k=0}^{\infty} s_{jk}z^{m_{j}}\bar{z}^{m_{k}}.
\]
The following definition extends definition \ref{def2:4} to the case where $b\neq 0$: 
\begin{definition}
We say that a complex manifold $(M,g)$ admits a local Kähler immersion into $F(N,b)$ if  given any point $p \in M$, there exists a neighbourhood $U$ of $p$ and a map $f: U \rightarrow  $F(N,b)$ $ such that:

\begin{itemize}
        \item[1.] f is holomorphic;
        
        \item[2.] f is isometric, i.e., $D_{p}^{M}(z)= D_{f(p)}^{b}(f(z)) $; 
        
        \item[3.] There exists $0< R< +\infty$ such that $\Sigma_{j=1}^{N} |f_{j}(z)|^{2} < R$.
    \end{itemize}

\end{definition}

\begin{definition}
    A real analytic Kähler manifold $(M,g)$ is $b$-resolvable of rank $N$ at $p\in M$ if the matrix $(s_{jk})$ is semipositive definite of rank $N$.
\end{definition}

In particular, $(M, g)$ is $1$-resolvable of rank $N$ at $p$ if the matrix of coefficients $(b_{jk})$ given by
\[
e^{D_{0}(z)} -1 = \sum_{j,k=0}^{\infty} b_{jk}z^{m_{j}}\bar{z}^{m_{k}},
\]
is positive semidefinite of rank $N$. 
Similarly $(M,g)$ is $-1$-resolvable of rank $N$ at $p$ if the matrix of coefficients $(c_{jk})$ given by
\[
1- e^{-D_{0}(z)} = \sum_{j,k=0}^{\infty} c_{jk}z^{m_{j}}\bar{z}^{m_{k}},
\]
is positive semidefinite of rank $N$. 

The significance  of the diastasis function is based on the following theorem, which summarizes results shown by Calabi in his renowned work \cite{Calabi}.

\begin{theorem}
\label{theo2:10}
    Let $(M,g)$ be a real analytic Kähler manifold. Then $(M,g)$ admits a local Kähler immersion into $\mathbb{C}P^{N}$(respectively $\mathbb{C}H^{N}$) if and only if the matrix $(b_{jk})$(respectively $(c_{jk})$) defined above  is positive semidefinite of rank at most N .
\end{theorem}

Similarly to the case of flat ambient space, we can state the following theorem.

\begin{theorem}(Global character of b-resolvability).
\label{theo2:11}
If a real analytic connected Kähler manifold $(M,g)$ is b-resolvable of rank $N$ at a point $p \in M$, then it also is at any other point.
\end{theorem}

The preceding theorem states that if a local Kähler immersion into $F(N, b)$ around a point $p \in M$ exists, then the same holds true for any other point. Based on this result, we conclude that a manifold is $b$-resolvable without specifying the point. A Kähler metric $g$ on a complex manifold $M$ is projectively induced if $(M,g)$ admits a holomorphic and isometric immersion into the complex projective space $\mathbb{C}P^{N}$, of dimernsion $N \leq +\infty$, endowed with its Fubini-Study metric $g_{FS}$, i.e., if there exists a holomorphic map $f: M 	\rightarrow  \mathbb{C}P^{N}$ such that $f^{*}g_{FS}=g$.
In particular, if $(M, g)$ is 1-resolvable, we also say that $g$ is projectively induced.

In this work, we are using the definition of ALE Kähler manifold according to  \cite{Joyce}. These spaces constitute the focus of our interest in this paper. A complete connected non-compact Riemannian $n$-manifold $(M, g)$ is said to be asymptotically locally Euclidean (or ALE) if the complement of a compact set $K$ consists of finitely many components, each of which is diffeomorphic to a quotient $(\mathbb{R}^{n} - D^{n}) / \Gamma_{j}$, where $\Gamma_{j} \subset O(n)$ is a finite subgroup which acts freely on the unit sphere, in such a way that $g$ again becomes the Euclidean metric plus error terms that fall off sufficiently rapidly at infinity.
In the folllowing, we have the actual formula for the mass in Kähler geometry. Since $M$ is a smooth manifold, one can define the compactly supported de Rham cohomology $H_{c}^{k}(M)$, as well as the usual de Rham cohomology. If $M$ is a complex manifold, it is in particular oriented, and Poincare duality therefore gives us an isomorphism $H_{c}^{2}(M) \cong [H^{2m-2}(M)]^{*}$. On the other hand, there is a natural map $H_{c}^{2}(M) \longrightarrow H^{2}(M)$ induced by the inclusion of compactly supported forms into all differential forms, and in the ALE setting, this map is actually an isomorphism. We may therefore define
\[
\clubsuit : H^{2}(M) \longrightarrow H_{c}^{2}(M)
\]
 to be its inverse. Using this notation, the explicit formula for the mass is given by the following:
 
 \begin{theorem}(\cite{Hein})
Any ALE Kähler manifold $(M,g,J)$ of complex dimension $m$ has mass given 
by\\
$$\mathfrak{m}(M,g)=-\frac{< \clubsuit (c_{1}), [\omega]^{m-1}>}{(2m-1)\pi^{m-1}} + \frac{(m-1)!}{4(2m-1)\pi^{m}}  \int_{M} s_{g} d\mu_{g},$$\\
 where $s_{g}$ and $d\mu_{g}$ are respectively the scalar curvature and volume form of $g$, 
while $c_{1}=c_{1}(M,J)\in H^{2}(M)$ is the first Chern class of the complex structure, $[\omega] \in H^{2}(M)$ is the Kähler class of $g$, and $<,>$ is the duality pairing between $H_{c}^{2}(M)$ and $H^{2m-2}(M).$
 \end{theorem}

\section{Kähler Immersions of ALE Kähler Metrics}
\label{sec:sec1}
\hspace{10pt} In this section, we present our main results. The first example of Ricci-flat (non-flat) Kähler metric constructed on a non-compact manifold is the Taub-NUT metric described by Claude Lebrun in \cite{Lebrun}. This is a $1$-parameter family of complex Kähler metrics on $\mathbb{C}^{2}$ defined by the Kähler potential
$$\Phi_{m}(u,v)=u^{2}+ v^{2}+ m(u^{4}+v^{4}),$$\\
where $u$ and $v$ are implicity given by $|z_{1}|=e^{m(u^{2}-v^{2})}u$, $|z_{2}|=e^{m(v^{2}-u^{2})}v$. In (\cite{Loi4}, page 522), Andrea Loi, Michela Zedda and Fabio Zuddas  investigate the projectively induced problem for the Taub-NUT metric on $\mathbb{C}^{2}$. 
\begin{lemma}(\cite{Loi4})
    Let $m \geq 0 $, $g_{m}$  be the Taub-NUT metric on $\mathbb{C}^{2}$ and $\alpha$ be a positive real number. Then $\alpha g_{m}$ is not projectively induced for $ m >\frac{ \alpha}{2}.$
\end{lemma}

In the following, we provide a proof of the Kähler immersion problem for the Taub-NUT metric into $\mathbb{C}^{N}$, which appeared in \cite{Loi4}. However, the details are not mentioned there.
%In the following, we want to show that the Taub-Nut metric $(C^{2}, \alpha_{m})$ cannot be Kähler immersed into the Complex hyperbolic space $C^{n}$
    \begin{example}
Prove that if the Taub-NUT metric ($\mathbb{C}^{2}, g_{m}$) admits a Kähler immersion into $\mathbb{C}^{N}$, $N  \leq  \infty$, then $m=0$.  
\end{example}

\begin{proof}
Assume that $( \mathbb{C}^{2},g_{m})$ adimts a Kähler immersion into $\mathbb{C}^{N}$. Then, it admits also a Kähler immersion into $\mathbb{C}^{N}$ of the Kähler submanifold of $( \mathbb{C}^{2}, \omega_{m})$ defined by $z_{2}=0$, $z_{1}=z$, endowed with the induced metric, having potential $\Phi_{m}=u^{2}+mu^{4}$. Then we have\\
\[
|z|^{2}=e^{2mu^{2}}u^{2}=(1+2mu^{2}+ O(u^{4}))u^{2}=u^{2}+2mu^{4}+O(u^{6}),
\]

\[
u^{2}=\frac{\sqrt{1+8m(|z|^{2}+O(|z|^{6}))}-1}{4m} =|z|^{2}-2m|z|^{4}+O(|z|^{6}).
\]
The power expansion of the diastasis function around the origin is defined as follows:\\
\[
\Phi_{m} = u^{2}+mu^{4}= |z|^{2}-2m|z|^{4}+O(|z|^{6}) + m (|z|^{2}-2m|z|^{4}+O(|z|^{6}))^{2}
\]

$$=|z|^{2}-m|z|^{4}-4m^{2}|z|^{6}+4m^{3}|z|^{8}+m(O(|z|^{12}) + O(|z|^{8})+ O(|z|^{10})).$$\\
According to Calabi’s criterion we have $-m\geq 0$, which implies that $m=0$.
    \end{proof}
    
Inspired by the above example, we aim to investigate the special case of  ALE Kähler metric without assuming any specific conditions on the curvature. The metric we describe in this paper is an example of ALE Kähler metric with Kähler potential $\Phi=\frac{1}{2}|z|^{2}+ a |z|^{-2\alpha}+O(|z|^{-2\alpha-1})$ where $\alpha >0$.
As proven by Hein and LeBrun in their work on mass in Kähler geometry \cite{Hein},  and the similarity of Kähler potential in (\cite{Arezzo}, section 3), studied by Arezzo, Vedova and Mazzieri, we can estimate the mass for these metrics, which is $a$.
Now, we present the main results of our work in the rest of the paper, where we investigate the Kähler immersions problem for our setting.

\begin{theorem}
    Consider a comlex manifold $M^{n}$ with ALE Kähler metric $g$ and Kähler potential 
  $$\Phi=\frac{1}{2}|z|^{2}+ a |z|^{-2\alpha}+O(|z|^{-2\alpha-1}),$$
where $\alpha >0$. The pair $(M^{n} , g)$ with positive mass does not admit a Kähler immersion into complex Euclidean space.
\end{theorem}

\begin{proof}
Firstly, we denote the error terms by $f(z,\bar{z})$, which is a power series with no terms of degree less than $-2\alpha -1$ in either the variables $z$ and $\bar{z}$. This means that there exist a positive constant $C$ and bounded function $H(z,\Bar{z})$ such that $H(z,\Bar{z})=\frac{f(z,\Bar{z})}{|z|^{-2\alpha -1}} \leq C$.
The diastasis function associated to these metrics is defined by: 
\begin{align*}
 D(z_{p},w_{p})&= \Phi(z_{p},\Bar{z}_p)+  \Phi(w_{p},\Bar{w}_p)- \Phi(z_{p},\Bar{w}_p)- \Phi(w_{p},\Bar{z}_p)\\
&=\frac{1}{2}(z_{p}+p)(\Bar{z}_p+p) + a(z_{p}+p)^{-\alpha}(\Bar{z}_p+p)^{-\alpha} + f(z_{p}+p,\Bar{z}_p+p)\\
&+\frac{1}{2}(w_{p}+p)(\Bar{w}_p+p) + a(w_{p}+p)^{-\alpha}(\Bar{w}_p+p)^{-\alpha} + f(w_{p}+p,\Bar{w}_p+p)\\
&-\frac{1}{2}(z_{p}+p)(\Bar{w}_p+p) - a(z_{p}+p)^{-\alpha}(\Bar{w_{p}}+p)^{-\alpha} - f(z_{p}+p,\Bar{w}_p+p)\\
&-\frac{1}{2}(w_{p}+p)(\Bar{z}_p+p) - a(w_{p}+p)^{-\alpha}(\Bar{z}_p+p)^{-\alpha} - f(w_{p}+p,\Bar{z}_p+p).\\
\end{align*}
 Next, we consider a fixed point $p=(R,R,...,R) \in \mathbb{C}^{n}$ for sufficiently large enough $R$. The diastasis function centered at $p$ is given by:
\begin{align*}
D_{0}(z_{p})&=\frac{1}{2}(z_{p}+p)(\Bar{z}_p+p) + a(z_{p}+p)^{-\alpha}(\Bar{z}_p+p)^{-\alpha} + f(z_{p}+p,\Bar{z}_p+p)\\
&+\frac{1}{2}|p|^{2}+ a|p|^{-2\alpha} +f(p,p)\\
&-\frac{1}{2}(z_{p}+p)(p) - a(z_{p}+p)^{-\alpha}(p)^{-\alpha}-f(z_{p}+p,p)\\
&-\frac{1}{2}(p)(\Bar{z_{p}} + p)-a (p)^{-\alpha}(\Bar{z}_p+p)^{-\alpha}-f(p,\Bar{z}_p+p ).\\
\end{align*}
Then, we consider the Taylor expansion of the diastasis function $D_{0}(z_{p})$ around $z_{p}=0$ as follows:
\begin{align*}
D_{0}(z_{p})& =\frac{n}{2}R^{2}+ \frac{R}{2}\sum_{i=1}^{n}z_{i,p} + \frac{R}{2}\sum_{i=1}^{n}\bar{z}_{i,p} 
+ \frac{1}{2}\sum_{i=1}^{n}z_{i,p}\bar{z}_{i,p} +anR^{-2\alpha} - a\alpha R^{-2\alpha-1}\sum_{i=1}^{n}\bar{z}_{i,p}\\
&\hspace{-12pt}- a\alpha R^{-2\alpha-1}\sum_{i=1}^{n}{z}_{i,p} + a\alpha^{2} R^{-2\alpha -2}\sum_{i=1}^{n}z_{i,p}\bar{z}_{i,p}+
    \frac{a\alpha (\alpha+1)}{2} R^{-2\alpha-2}\sum_{i=1}^{n}z_{i,p}^{2}\\
&\hspace{-12pt}+ \frac{a\alpha (\alpha+1)}{2} R^{-2\alpha-2}\sum_{i=1}^{n}\bar{z}_{i,p}^{2}- \frac{a\alpha^{2}(\alpha+1)}{2}R^{-2\alpha-3} \sum_{i=1}^{n}\bar{z}_{i,p}^{2} z_{i,p}+ \frac{a\alpha^{2}(\alpha+1)}{2}R^{-2\alpha-3} \sum_{i=1}^{n}\bar{z}_{i,p} z_{i,p}^{2}\\
&\hspace{-12pt}+\frac{a\alpha^{2}(\alpha+1)^{2}}{4} R^{-2\alpha-4 }\sum_{i=1}^{n}\bar{z}_{i,p}^{2} z_{i,p}^{2}+...+ \frac{n}{2}R^{2}+ anR^{-2\alpha}+...\\
&\hspace{-12pt}-\frac{n}{2}R^{2}- naR^{-2\alpha}-\frac{R}{2}\sum_{i=1}^{n}z_{i,p}+ a\alpha R^{-2\alpha-1}\sum_{i=1}^{n}z_{i,p} - \frac{a\alpha (\alpha +1)}{2}R^{-2\alpha-2}\sum_{i=1}^{n}z_{i,p}^{2}+...\\\
&\hspace{-12pt}-\frac{n}{2}R^{2}- naR^{-2\alpha}-\frac{R}{2}\sum_{i=1}^{n}\bar{z}_{i,p}+ a\alpha R^{-2\alpha-1}\sum_{i=1}^{n}\bar{z}_{i,p} - \frac{a\alpha (\alpha +1)}{2}R^{-2\alpha-2}\sum_{i=1}^{n}\bar{z}_{i,p}^{2}+...\\\ 
\\
&\hspace{-12pt} + f(z_{p}+p,\Bar{z}_{p}+{p})-f(p,p) -f(z_{p}+p,p) -f(p,\Bar{z}_{p}+p).\\\
\end{align*}

Now considering the matrix of the coefficients of the diastasis function, excluding the error terms, denoted by $A_{ij}$, the coefficients of the whole matrix is thus $A_{ij}+H_{ij}A_{ij}R^{-\frac{1}{2}}$. Clearly, if $A_{ij}$ is not positive semidefinite, then $A_{ij}+H_{ij}A_{ij}R^{-\frac{1}{2}}$ is not positive semidefinite as well.
Therefore we get,
\begin{align*}
D_{0}(z_{p}) - \text{error terms}&=(\frac{1}{2}+a\alpha^{2}R^{-2\alpha-2})z_{1,p}\bar{z}_{1,p}- \frac{a\alpha^{2}(\alpha+1)}{2}  R^{-2\alpha-3}\bar{z}_{1,p}^{2}z_{1,p} \\
& - \frac{a\alpha^{2}(\alpha+1)}{2} R^{-2\alpha-3}\bar{z}_{1,p}z_{1,p}^{2} +\frac{a\alpha^{2}(\alpha+1)^{2}}{4} R^{-2\alpha-4} \bar{z}_{1,p}^{2}z_{1,p}^{2}\\
&+(\frac{1}{2}+a\alpha^{2}R^{-2\alpha-2})z_{2,p}\bar{z}_{2,p}+...\\
\end{align*}
$A_{ij}$ is $ \infty \times \infty$ matrix. We consider the first $4 \times 4$ matrix of the coefficients as follows:
\medskip

\[
\begin{bmatrix}
    0 & 0 & 0 &0  \\
    0& a_{11} & 0 & a_{13}  \\
    0 & 0 & a_{22}& 0  \\
    0 & a_{31} & 0 & a_{33}
\end{bmatrix}
\]
\medskip

\hspace{-15pt}where $a_{11}=a_{22}=\frac{1}{2}+a\alpha^{2}R^{-2\alpha-2}, a_{13}= a_{31}=-\frac{a\alpha^{2}(\alpha+1)}{2}  R^{-2\alpha-3}$ and $a_{33}=\frac{a\alpha^{2}(\alpha+1)^{2}}{4} R^{-2\alpha-4}$.\\
If $(M^{n}, g)$   admits a Kähler immersion into $\mathbb{C}^{N}$, then it is resolvable of rank at most $N$. This implies that the matrix of coefficients is positive semidefinite and has rank at most $N$. Consequently, we obtain the following conditions:
\[
\frac{1}{2}+a\alpha^{2}R^{-2\alpha-2} \geq 0, \text{ and }   \frac{a}{2}+a^{2}\alpha^{2}R^{-2\alpha -4} \geq 0.
\]
We can easily see that $ a \geq 0$. The above matrix is not semipositive definite for any negative value of $a$; indeed, it has a negative eigenvalue. Hence, the whole matrix of the coefficients cannot be positive semidefinite and the diastasis function is not resolvable of any rank. Therefore, using Theorems \ref{theo2:6}  and \ref{theo2:7} , we conclude  that  $(M^{n} , g)$ does not admit a Kähler immersion into complex Euclidean space for any negative value of $a$.
\end{proof}

Now, we are considering the same problem argument for complex hyperbolic space.

\begin{theorem}
    Consider a complex manifold $M^{n}$ with ALE Kähler metric $g$ and Kähler potential 
    $$\Phi=\frac{1}{2}|z|^{2}+ a |z|^{-2\alpha}+O(|z|^{-2\alpha-1}),$$ 
where $\alpha >0$. The pair $(M^{n} , g)$ with any sign of mass cannot be Kähler immersed
    into complex hyperbolic space.
\end{theorem}
\begin{proof}
    In the proof above, we computed the diastasis function associated to these metrics. Now,  $(M^{n} , g)$ is -1-resolvable of rank $N$ if the matrix of the cofficients $(b_{ij})$
    given by
    \[
    1-e^{-D_{0}(z_{p})}= \sum_{i,j=0}^{\infty} b_{ij}z^{m_{i}}\bar{z}^{m_{j}},
    \]
    is positive semidefinite of rank $N$. Next, we consider the power expansion of the function $1- e^{-(D_{0}(z_{p}) - \text{error terms})}$ around the origin as follows:
\begin{align*}
1- e^{-(D_{0}(z_{p}) - \text{error terms})}&=D_{0}(z_{p}) -\frac{D_{0}(z_{p})^2}{2!} + \frac{D_{0}(z_{p})^{3}}{3!} - \frac{D_{0}(z_{p})^{4}}{4!} + ...
\\
&=(\frac{1}{2}+a\alpha^{2}R^{-2\alpha-2})z_{1,p}\bar{z}_{1,p}- \frac{a\alpha^{2}(\alpha+1)}{2}  R^{-2\alpha-3}\bar{z}_{1,p}^{2}z_{1,p} \\
&+( \frac{a\alpha^{2}(\alpha+1)^{2}}{4} R^{-2\alpha-4}  -\frac{1}{2} (\frac{1}{2}+a\alpha^{2}R^{-2\alpha-2})^{2})\bar{z}_{1,p}^{2}z_{1,p}^{2}\\
&-\frac{a\alpha^{2}(\alpha+1)}{2} R^{-2\alpha-3}\bar{z}_{1,p}z_{1,p}^{2}+(\frac{1}{2}+a\alpha^{2}R^{-2\alpha-2})z_{2,p}\bar{z}_{2,p}+... \\
\end{align*}   
Now, the coefficients of the first $4 \times 4$ matrix are zero except for the following terms:

$$b_{11}=b_{22}=\frac{1}{2}+a\alpha^{2}R^{-2\alpha-2}, \text{ } \text{ } b_{13}= b_{31}=-\frac{a\alpha^{2}(\alpha+1)}{2}  R^{-2\alpha-3},$$ 
$$\text{and } b_{33}= \frac{a\alpha^{2}(\alpha+1)^{2}}{4} R^{-2\alpha-4}  -\frac{1}{2} (\frac{1}{2}+a\alpha^{2}R^{-2\alpha-2})^{2}.$$

\hspace{-15pt}With the same argument as in the previous proof and using Calabi’s criterion, we obtain the following conditions:
\[
\frac{1}{2}+a\alpha^{2}R^{-2\alpha-2} \geq 0, \text{ and } \frac{a\alpha^{2}(\alpha+1)^{2}}{8}R^{-2\alpha-4} -\frac{1}{2}(\frac{1}{2}+a\alpha^{2}R^{-2\alpha-2})^{3} \geq 0.
\]
We can easily see that the second condition is always negative for any value of $a$.
Consequently, we conclude that the matrix mentioned above is not positive semidefinite. Therefore, using Theorems \ref{theo2:10} and \ref{theo2:11} , we conclude that $(M^{n} , g)$ does not admit a Kähler immersion into complex hyperbolic space for any value of $a$.

 \end{proof}

We now turn our attention to the problem in the context of complex projective space. Employing the same method, we find that it is not possible to derive any conclusive results for complex projective space under the same assumptions as in Theorem 3.4. We consider the power expansion of the function $ e^{D_{0}(z_{p})}-1$ around the origin as follows:

\begin{align*}
e^{D_{0}(z_{p})}-1&=D_{0}(z_{p}) + \frac{D_{0}(z_{p})^2}{2!} + \frac{D_{0}(z_{p})^{3}}{3!} + \frac{D_{0}(z_{p})^{4}}{4!} + ...
\\
&=(\frac{1}{2}+a\alpha^{2}R^{-2\alpha-2})z_{1,p}\bar{z}_{1,p}- \frac{a\alpha^{2}(\alpha+1)}{2}  R^{-2\alpha-3}\bar{z}_{1,p}^{2}z_{1,p} - \frac{a\alpha^{2}(\alpha+1)}{2} R^{-2\alpha-3}\bar{z}_{1,p}z_{1,p}^{2}\\
&+(  (\frac{1}{2}+a\alpha^{2}R^{-2\alpha-2})^{2}+ \frac{a\alpha^{2}(\alpha+1)^{2}}{4}) R^{-2\alpha-4} \bar{z}_{1,p}^{2}z_{1,p}^{2}+(\frac{1}{2}+a\alpha^{2}R^{-2\alpha-2})z_{2,p}\bar{z}_{2,p}+... \\
\end{align*}   
 The nonzero coefficients of the first $4 \times 4$ matrix are the following terms:
 
$$b_{11}=b_{22}=\frac{1}{2}+a\alpha^{2}R^{-2\alpha-2}, \text{ } \text{ } b_{13}= b_{31}=-\frac{a\alpha^{2}(\alpha+1)}{2}  R^{-2\alpha-3},$$ 
$$\text{and } b_{33}= \frac{1}{2}(\frac{1}{2}+a\alpha^{2}R^{-2\alpha-2})^{2}+ \frac{a\alpha^{2}(\alpha+1)^{2}}{4}R^{-2\alpha-4}.$$\\
 Using the same argument as above, we obtain the following conditions:
    \[
    \frac{1}{2}+a\alpha^{2}R^{-2\alpha-2} \geq 0,
    \]
    \[
    \frac{1}{8}a\alpha^{2}(\alpha+1)^{2}R^{-2\alpha-4}+ \frac{1}{16}+ \frac{3}{8} a\alpha^{2}R^{-2\alpha-2}+\frac{3}{4}a^{2}\alpha^{4}R^{-4\alpha-4} + \frac{1}{2}a^{3}\alpha^{6}R^{-6\alpha-6} \geq 0.
    \]\\
    The above conditions hold for any value of $a$. However, from the first $4 \times 4$ matrix, we cannot obtain any result.

The Simanca metric on the blow-up $\mathbb{\Tilde {C}}^{2}$ of $\mathbb{C}^{2}$ at the origin is a well-known and important exmaple (both mathematically and physically) of non homogeneous complete, zero constant scalar curvature metric. Furthermore, in \cite{simanca} the authors proved that Simanca metric is projectively induced.
\begin{remark}
The Simanca metric is a well-known counterexample which has positive mass and is projectively induced \cite{simanca}.
\end{remark}

\section*{Acknowledgement}
We thank Professor Claudio Arezzo for his insightful discussions and valuable comments on this work. We also would like to thank Professor Reza Seyyedali for several useful discussions and helpful advice.

\bibliographystyle{plain}

\newcommand{\etalchar}[1]{$^{#1}$}

\begin{itemize}
    \item Tarbiat Modares University, Tehran and ICTP, Trieste; farnazghanbari@modares.ac.ir
    \item Tarbiat Modares University, Tehran; aheydari@modares.ac.ir
\end{itemize}
\end{document}